\documentclass[reqno]{amsart}
\usepackage{hyperref}
\def\N{{\mathbb{N}}}

\def\R{{\mathbb{R}}}

\begin{document}
\title[Pontryagin principles]{Pontryagin principles in infinite horizon in presence of asymptotical constraints}

\author[BLOT and NGO]
{Jo\"{e}l BLOT and Thoi Nhan NGO}

\address{Jo\"{e}l Blot: Laboratoire SAMM EA 4543,\newline
Universit\'{e} Paris 1 Panth\'{e}on-Sorbonne, centre P.M.F.,\newline
90 rue de Tolbiac, 75634 Paris cedex 13,
France.}
\email{blot@univ-paris1.fr}
\address{Thoi Nhan Ngo: Laboratoire SAMM EA 4543,\newline
Universit\'{e} Paris 1 Panth\'{e}on-Sorbonne, centre P.M.F.,\newline
90 rue de Tolbiac, 75634 Paris cedex 13,
France.}
\email{ngothoinhan@gmail.com}
\date{November 5, 2015}
\begin{abstract} We establish necessary conditions of optimality for discrete-time infinite-horizon optimal control in presence of constraints at infinity. These necessary conditions are in form of weak and strong Pontryagin principles. We use a functional analytic framework and multipliers rules in Banach (sequence) spaces. We establish new properties on Nemytskii operators in sequence spaces. We also provide sufficient conditions of optimality.
\end{abstract}

 \maketitle
\numberwithin{equation}{section}
\newtheorem{theorem}{Theorem}[section]
\newtheorem{lemma}[theorem]{Lemma}
\newtheorem{example}[theorem]{Example}
\newtheorem{remark}[theorem]{Remark}
\newtheorem{definition}[theorem]{Definition}
\newtheorem{corollary}[theorem]{Corollary}
\newtheorem{proposition}[theorem]{Proposition}

\section{Introduction}
The aim of this paper is to establish necessary conditions of optimality in the form of Pontryagin principles for the following Optimal Control problem
\[
(P)
\left\{
\begin{array}{rl}
{\rm Maximize} & K(\underline{y}, \underline{u}) := \sum_{t = 0}^{+ \infty} \beta^t  \psi(y_t,u_t)\\
{\rm when} & \underline{y} :=(y_t)_{t \in \N} \in (\R^n)^{\N},  \underline{u} :=(u_t)_{t \in \N} \in U^{\N}\\
\null &  y_0 = \eta, \; \lim_{t \rightarrow + \infty} y_t = y_{\infty}, \; \underline{u} \; {\rm is} \; {\rm bounded}\\
\null & \forall t \in \N, y_{t+1} = g(y_t, u_t)
\end{array}
\right.
\]
where $\beta \in (0,1)$, $U \subset \R^d$ is nonempty, $\psi : \R^n \times U \rightarrow \R$ is a function, $\eta$ and $y_{\infty}$ are fixed vectors of $\R^n$, $g : \R^n \times U \rightarrow \R^n$ is a function, and $(\R^n)^{\N}$ (respectively $U^{\N}$) denotes the set of the sequences in $\R^n$ (respectively $U$). In comparison with existing results on bounded processes, the specificity of the present work is the presence of the asymptotical constraint on the state variable: $\lim_{t \rightarrow + \infty} y_t = y_{\infty}$; its meaning is that the optimal state of the problem stays near a "good" state value on the long run.
\vskip2mm
Such problem in discrete time and infinite horizon arises in several fields of applications, for instance in optimal growth macroeconomic theory and in optimal management of forests and fisheries; see the references in \cite{BH2}.
\vskip1mm
Our approach is functional analytic; we translate our problems as static of optimization in suitable Banach sequence spaces.  
\vskip2mm
Now we describe the contents of the paper. In Section 2 we introduce a problem of optimal control which is equivalent to the initial problem in order to use classical sequence spaces: $c_0(\N,\R^n)$ the space of the sequences into $\R^n$ which converge to zero at infinity, and ${\ell}^{\infty}(\N,U)$  the space of the sequences into $U$ which are bounded.\\
In Section 3 we study properties of operators and functionals on sequence spaces. A first novelty is a characterization of the operators which send $c_0(\N,\R^n) \times {\ell}^{\infty}(\N,U)$ into $c_0(\N,\R^n)$ (Theorem \ref{th31}). The other results use this characterization and existing results on Nemytskii operators from $ {\ell}^{\infty}(\N,\R^n) \times  {\ell}^{\infty}(\N,U)$ into $ {\ell}^{\infty}(\N,\R^m)$. \\
Section 4 is devoted to the solutions which converge toward zero of a linear difference equation. These results are useful to establish regularity properties of the differential of operators which formalize the nonlinear difference equation which governs the system.\\
In Section 5 we establish a variation of a Karush-Kuhn-Tucker theorem which is useful for weak Pontryagin principles and we recall a result which is useful for strong Pontryagin principles.
\vskip1mm
\noindent
In Section 6 and Section 7 (respectively Section 8 and Section 9) we establish weak (respectively strong) Pontryagin principles. \\
In Section 10 and Section 11, we establish results of sufficient condition of optimality.
\section{An equivalent problem}
In this section we formulate a problem which is equivalent to Problem (P) for which we can work in classical Banach sequence spaces.\\
We consider the following Optimal Control problem
\[
( P1)
\left\{
\begin{array}{rl}
{\rm Maximize} & J(\underline{x}, \underline{u}) := \sum_{t = 0}^{+ \infty} \beta^t  \phi(x_t,u_t)\\
{\rm when} & \underline{x} :=(x_t)_{t \in \N} \in c_0(\N,\R^n),  \underline{u} :=(u_t)_{t \in \N} \in {\ell}^{\infty}(\N,U)\\
\null &  x_0 = \sigma\\
\null & \forall t \in \N, x_{t+1} = f(x_t, u_t).
\end{array}
\right.
\]
When we choose $\phi : \R^n \times U \rightarrow \R$ as $\phi(x,u) = \psi(x + y_{\infty},u)$, $f(x,u) = g(x + y_{\infty}, u) - y_{\infty}$, $x_t = y_t - y_{\infty}$ for all $t \in \N$, $\sigma = \eta - y_{\infty}$, Problem $({\mathcal P1})
$ is equivalent to Problem $({\mathcal P})$. And so our strategy for the sequel of the paper is to work on (P1) and to translate the results on $({\mathcal P1})$ into results on (P).\\
For the properties of $c_0(\N, \R^n)$ we refer to Section 15.3 in \cite{AB}, and for those of the space ${\ell}^{\infty}(\N, U)$ we refer to Section 15.7 in \cite{AB}.
\section{Nonlinear operators and functionals}
This section is devoted to the study of several operators between sequence spaces; notably the Nemytskii operators (also called superposition operators), and to the study of the functionals which define the criterium of our maximization problems. We establish results of continuity and of Fr\'echet differentiability.
\vskip1mm
\begin{theorem}\label{th31}
Let $X$, $V$, $W$ be three real normed spaces, $U$ be a nonempty subset of $V$, and  $F : X \times U \rightarrow W$ be a mapping such that, for all $x \in X$, the partial mapping $F(x, \cdot)$ transforms the bounded subsets of $U$ into bounded subsets of $W$. Then the following assertions are equivalent.
\begin{enumerate}
\item [(i)] $\forall \underline{x} \in c_0(\N, X)$, $\forall \underline{u} \in  {\ell}^{\infty}(\N, U)$, $(F(x_t,u_t))_{t \in \N} \in c_0(\N,W)$.
\item[(ii)] For all nonempty bounded subset $B$ in $U$, $\lim_{x \rightarrow 0}( \sup_{u \in B} \Vert F(x, u) \Vert) = 0$.
\end{enumerate}
\end{theorem}
\begin{proof} ${\bf (i \Longrightarrow ii)}$ Let $B$ be a nonempty bounded subset of $U$. Let $\underline{x} \in c_0(\N, X)$. From the assumption on $F$, we know that, for all $t \in \N$, we have $\sup_{u \in B} \Vert F(x_t, u) \Vert < + \infty$. Therefore, for all $t \in \N$, there exists $u_t \in B$ such that 
$$0 \leq \sup_{u \in B} \Vert F(x_t, u) \Vert \leq \Vert F(x_t, u_t) \Vert + \frac{1}{t+1}.$$
Since, for all $t \in \N$, $u_t \in B$, we have $\underline{u} \in {\ell}^{\infty}(\N, U)$. Then using (i), we obtain \\
$\lim_{t \rightarrow + \infty} \Vert F(x_t,u_t) \Vert = 0$, and from the previous inequality we obtain \\
 $\lim_{t \rightarrow + \infty} (\sup_{u \in B} \Vert F(x_t, u) \Vert) = 0$, and since we work in normed spaces we can use the sequential characterization of the limit and assert that we obtain (ii).\\
${\bf (ii \Longrightarrow i)}$ Let $\underline{x} \in c_0(\N, X)$ and $\underline{u} \in  {\ell}^{\infty}(\N, U)$. Then the subset $B := \{ u_t : t \in \N \}$ is bounded, and, for all $t \in \N$, the following inequality holds:
$$0 \leq \Vert F(x_t, u_t) \Vert \leq  \sup_{u \in B} \Vert F(x_t, u) \Vert,$$
and from (ii), since $\lim_{t \rightarrow + \infty} x_t = 0$, we obtain $\lim_{t \rightarrow + \infty} (\sup_{u \in B} \Vert F(x_t, u) \Vert) = 0$, and from the previous inequality we deduce $\lim_{t \rightarrow + \infty} \Vert F(x_t,u_t) \Vert = 0$, i.e. the  sequence $(F(x_t,u_t))_{t \in \N}$ belongs to $ c_0(\N,Y)$.
\end{proof} 
\begin{remark}\label{rem32} Assertion (i) of Theorem \ref{th31} permits to define the Nemytskii operator
$$N_F :  c_0(\N, X) \times {\ell}^{\infty}(\N, U) \rightarrow c_0(\N,W), N_F((x_t)_{t \in \N}, (u_t)_{t \in \N}) := (F(x_t,u_t))_{t \in \N}.$$
\end{remark}
\begin{remark}\label{rem33}
We set $B_R := \{ v \in V : \Vert v \Vert \leq R \}$ when $R \in (0, + \infty)$.
In the setting of Theorem \ref{th31}, the assumption on $F$ is equivalent to the following condition: $\forall x \in X$, $\forall R \in (0, + \infty)$, $\sup_{ u \in B_R \cap U} \Vert F(x, u) \Vert < + \infty$, and the assertion (ii) is equivalent to:  $\forall R \in (0, + \infty)$, $\lim_{x \rightarrow 0} (\sup_{u \in B_R \cap U } \Vert F(x, u) \Vert) = 0$. \\
Also note that assumption (ii) and the continuity of $F( \cdot, u)$ for all $u \in U$ imply $F(0,u) = 0$ for all $u \in U$, since, for all $u \in \R^d$, $\{ u\}$ is a nonempty bounded subset and $0 = \lim_{x \rightarrow 0} \Vert F(x,u) \Vert = \Vert F(0,u) \Vert$.
\end{remark}
\begin{remark}\label{rem34} In the setting of Theorem \ref{th31}, if in addition we assume that dim$V < + \infty$ and $U$ is closed, using the relative compactness of bounded subsets of $U$, if $F(x, \cdot) \in  C^0( U,W)$ (the space of continuous mappings from $U$ into $W$), $F(x, \cdot)$ transforms the bounded sets into bounded sets.
\end{remark}
\begin{theorem}\label{th35} Let $U$ be a nonempty closed subset of $\R^d$. Let $F \in C^0(\R^n \times U, \R^m)$ such that $\lim_{x \rightarrow 0}( \sup_{u \in B_R \cap U } \Vert F(x, u) \Vert ) = 0$ for all $R \in (0, + \infty)$.\\
Then we have the continuity of the Nemytskii operator on $F$, i.e.\\
 $N_F \in C^0(c_0(\N, \R^n) \times {\ell}^{\infty}(\N, U), c_0(\N,\R^m))$.
\end{theorem}
\begin{proof}
First using Remark \ref{rem34}, the assumption of Theorem \ref{th31} is fulfilled. Using Remark \ref{rem33}, assertion (ii) of Theorem \ref{th31} is fulfilled, and using Theorem \ref{th31}, and Remark \ref{rem32}, the operator $N_F$ is well defined from $c_0(\N, \R^n) \times {\ell}^{\infty}(\N, U)$ into $c_0(\N,\R^m)$.
Since the bounded subsets of $\R^n \times U$ are relatively compacts, we can defined the other Nemytskii operator 
$$N^1_F : {\ell}^{\infty}(\N, \R^n) \times {\ell}^{\infty}(\N, U) \rightarrow {\ell}^{\infty}(\N,\R^m), N^1_F((x_t)_{t \in \N}, (u_t)_{t \in \N}) := (F(x_t,u_t))_{t \in \N}.$$
Since ${\ell}^{\infty}(\N, \R^n) \times {\ell}^{\infty}(\N, U)$ is isometrically isomorphic to 
${\ell}^{\infty}(\N, \R^n \times U)$, as a consequence of Theorem A1.1 in \cite{BCr} (p. 22), we can assert that $N^1_F$ is continuous, and then $N_F$ is continuous as a restriction of a continuous operator.
\end{proof}
Let $X$, $V$, $W$ be real Banach spaces, and $U$ be a nonempty subset of $V$. Let $F : X \times U \rightarrow W$ be a mapping. we say that $F$ is of class $C^1$ on $ X \times U$ when there exist an open subset $U_1$ in $ V$ such that $U \subset U_1$ and a mapping $F_1 \in C^1(X \times U_1 ,  W)$ such that ${F_1}_{\mid_{ X \times U}} = F$. Such a definition is common in the differential theories; see e.g. \cite{Mi} (p. 1).
\begin{remark}\label{rem36}
When $F_1, F_2 \in C^1(X \times U_1 , W)$ such that ${F_1}_{\mid_{ X \times U}}  = {F_2}_{\mid_{ X \times U}} = F$, when $U$ is star-shaped with respect to $u^0$, when $u, u^0 \in U$ and $x^0 \in X$, note that, for all $\theta \in (0,1)$, we have $F_1(x^0, (1 - \theta) u^0 + \theta u) =  F_2(x^0, (1 - \theta) u^0 + \theta u) = F (x^0, (1 - \theta) u^0 + \theta u)$. Therefore we have $D_2 F_1(x^0,u^0)(u - u^0) ={ \frac{d}{d \theta}}_{\mid_{\theta = 0}}F_1(x^0, (1 - \theta) u^0 + \theta u) = { \frac{d}{d \theta}}_{\mid_{\theta = 0}}F_2(x^0, (1 - \theta) u^0 + \theta u)= D_2 F_2(x^0, u^0)(u-u^0)$, and so $D_2F(x^0,u^0)(u-u^0)$ does not depend of the extension of $F$.\\
Recall that $U$ is star-shaper with respect to $u^0$ means that, for all $u \in U$, the segment $[u^0,u]$ is included in $U$, \cite{Sp} (p. 93).
\end{remark}
When $V$ and $W$ are normed spaces, ${\mathfrak L}(V,W)$ denotes the space of the linear continuous functions from $V$ into $W$, and when $L \in {\mathfrak L}(V,W)$, we write $\Vert L \Vert_{\mathfrak L} := \sup \{ \Vert Lv \Vert : v \in V, \Vert v \Vert \leq 1 \}$.
\vskip1mm
\begin{theorem}\label{th37}
Let $U$ be a nonempty closed subset of $\R^d$.
Let $F : \R^n \times U \rightarrow \R^m$ be a mapping  which satisfies the following assumptions.
\begin{enumerate}
\item[(i)] $F \in C^1(\R^n \times U , \R^m)$.
\item[(ii)] There exists $u^0 \in U$ such that $F(0,u^0) = 0$ and $U$ is star-shaped with respect to $u^0$.
\item[(iii)] $\lim_{x \rightarrow 0} (\sup_{u \in B} \Vert DF(x,u) \Vert_{\mathfrak L} ) = 0$ for all nonempty bounded subset $B \subset U$.
\end{enumerate}
Then $N_F \in C^1(c_0(\N, \R^n) \times {\ell}^{\infty}(\N, U), c_0(\N,\R^m))$, and for all $\underline{x} \in c_0(\N, \R^n)$, $\underline{u} \in {\ell}^{\infty}(\N, U)$, $\underline{\delta x} \in c_0(\N, \R^n) $, $\underline{\delta u} \in {\ell}^{\infty}(\N, \R^d)$, we have 
\[
\begin{array}{ccl}
DN_F(\underline{x}, \underline{u})(\underline{\delta x},\underline{ \delta u}) & = & (DF(x_t,u_t)(\delta x_t, \delta u_t))_{t \in \N}\\
\null & = & (D_1F(x_t,u_t)\delta x_t + D_2F(x_t,u_t)\delta u_t)_{t \in \N}
\end{array}
\]
where $D_1$ and $D_2$ denote the partial Fr\'echet differentiations.
\end{theorem}
\begin{proof}
Let $B \subset U$ be nonempty and bounded. Let $R \in (0, + \infty)$ such that $\Vert u \Vert \leq R$ when $u \in B \cup \{ u^0 \}$. Using the mean value theorem, we have, for all $x \in \R^n$ and for all $u \in B$,
\[
\begin{array}{rcl}
\Vert F(x,u) \Vert & \leq & \Vert F(x,u^0) \Vert + \sup_{v \in [u^0,u]} \Vert D_2 F(x,v) \Vert \cdot \Vert v \Vert\\
\null & \leq & \Vert F(x,u^0) \Vert + \sup_{ v \in B_R \cap U} \Vert D_2F (x,v) \Vert \cdot \Vert v \Vert\\
\null & \leq & \Vert F(x,u^0) \Vert + \sup_{ v \in B_R \cap U} \Vert D F(x,v) \Vert \cdot R
\end{array}
\]
which implies
$$\sup_{u \in B}\Vert F(x,u) \Vert \leq \Vert F(x,u^0) \Vert + \sup_{ v \in B_R \cap U} \Vert D F(x,v) \Vert \cdot R,$$
and therefore, using assumptions (iii) and  (ii) and the continuity of $F$, we obtain
\begin{equation}\label{eq31}
\lim_{ x \rightarrow 0} (\sup_{u \in B}\Vert F(x,u) \Vert ) = 0 \; {\rm when} \; B \neq \emptyset, B \subset \R^d \; {\rm is} \; {\rm bounded}.
\end{equation}
Since $F$ is continuously Fr\'echet differentiable, $F$ is continuous, and then, with (\ref{eq31}), we can apply Theorem \ref{th35} to $F$ and assert that $N_F$ is well defined from $c_0(\N, \R^n) \times {\ell}^{\infty}(\N, U)$ into $c_0(\N,\R^m)$ and it is continuous, i.e.
\begin{equation}\label{eq32} 
N_F \in C^0(c_0(\N, \R^n) \times {\ell}^{\infty}(\N, U), c_0(\N,\R^m)).
\end{equation}
Using Theorem A1.2 of \cite{BCr} (p. 24) to the operator $N^1_F$ defined in the proof of the previous theorem, we can assert that $N^1_F$ is continuously Fr\'echet differentiable from ${\ell}^{\infty}(\N, \R^n) \times {\ell}^{\infty}(\N, U)$ into ${\ell}^{\infty}(\N, \R^m)$ and, for all $\underline{x} \in {\ell}^{\infty}(\N, \R^n)$, $\underline{u} \in {\ell}^{\infty}(\N, U)$, $ \underline{ \delta x} \in {\ell}^{\infty}(\N, \R^n) $, $\underline{ \delta u} \in {\ell}^{\infty}(\N, \R^d)$, we have $
DF(\underline{x}, \underline{u})( \underline{ \delta x},  \underline{ \delta u})  =  (DF(x_t,u_t)(\delta x_t, \delta u_t))_{t \in \N}$. Since $N_F$ is a restriction to a vector subspace of $N^1_F$, we obtain that $N_F$ is continuously Fr\'echet differentiable from $c_0(\N, \R^n) \times {\ell}^{\infty}(\N, U)$ into $c_0(\N,\R^m)$ 
and the formula of its differential is identical to this one of $N^1_F$.
\end{proof}
The following result is useful to translate the properties of the dynamical system which governs (P1) into the language of operators between sequence spaces.
\begin{corollary}\label{cor38}
Let $U$ be a nonempty closed subset of $\R^d$.
Let $f : \R^n \times U \rightarrow \R^n$ be a mapping which satisfies the assumptions (i, ii, iii) of Theorem \ref{th37}.
We consider the operator ${\mathcal T}(\underline{x}, \underline{u}) := (x_{t+1} -f(x_t,u_t))_{t \in \N}$. \\
 Then ${\mathcal T} \in C^1(c_0(\N, \R^n) \times {\ell}^{\infty}(\N, U), c_0(\N, \R^n))$ and for all $\underline{x} \in c_0(\N, \R^n)$, $\underline{u} \in {\ell}^{\infty}(\N, U)$, $\underline{\delta x} \in c_0(\N, \R^n) $, $\underline{\delta u} \in {\ell}^{\infty}(\N, \R^d)$, we have 
\[
\begin{array}{rcl}
D{\mathcal T}(\underline{x}, \underline{u})(\underline{\delta x}, \underline{\delta u}) & = & (\delta x_{t+1} -Df(x_t,u_t)(\delta x_t, \delta u_t) )_{t \in \N}\\
\null & = & (\delta x_{t+1} - D_1f(x_t,u_t)\delta x_t - D_2f(x_t,u_t) \delta u_t )_{t \in \N}.
\end{array}
\]
\end{corollary}
\begin{proof}
Since $f$ satisfies the assumptions of Theorem \ref{th37}, we have $N_f \in C^1(c_0(\N, \R^n) \times {\ell}^{\infty}(\N, U), c_0(\N, \R^n))$. We set $\Lambda (\underline{x}, \underline{u}) := (x_{t+1})_{t \in \N}$ when $\underline{x} \in c_0(\N, \R^n)$ and $\underline{u} \in {\ell}^{\infty}(\N, U)$. Then $\Lambda$ is well defined and is linear. Since 
$\Vert \Lambda (\underline{x}, \underline{u}) \Vert_{\infty} \leq   \Vert \underline{x} \Vert_{\infty} \leq  \Vert \underline{x} \Vert_{\infty} +  \Vert \underline{u} \Vert_{\infty}$, $\Lambda$ is continuous, and consequently it is of class $C^1$. Moreover we have for all $\underline{x} \in c_0(\N, \R^n)$, $\underline{u} \in {\ell}^{\infty}(\N, U)$, $\underline{\delta x} \in c_0(\N, \R^n) $, $\delta\underline{u} \in {\ell}^{\infty}(\N, \R^d)$, $D \Lambda (\underline{x}, \underline{u}) (\underline{\delta x},  \underline{\delta u}) = \Lambda (\underline{\delta x}, \underline{\delta u}) = ( \delta x_{t+1})_{t \in \N}$. Note that ${\mathcal T} =  \Lambda - N_f $ which implies that ${\mathcal T}$ is continuously Fr\'echet differentiable, and using Theorem \ref{th37} we obtain
\[
\begin{array}{rcl}
D {\mathcal T} (\underline{x}, \underline{u}) (\underline{\delta x}, \underline{\delta u}) &=& D \Lambda (\underline{x}, \underline{u}) (\underline{\delta x}, \underline{\delta u}) -D N_f (\underline{x}, \underline{u}) (\underline{\delta x}, \underline{\delta u})   \\
\null &=& ( \delta x_{t+1} - Df(x_t,u_t)(\delta x_t, \delta u_t) )_{t \in \N}.
\end{array}
\]
\end{proof}
\begin{remark}\label{rem39} We consider the operator ${\mathcal F} : c_0(\N_*, \R^n) \rightarrow c_0(\N, \R^n)$ defined by ${\mathcal F} ( \underline{x'}) := \underline{x}$ where $x_0 := 0$ and $x_t := x'_t$ when $t \in \N_*$. We introduce the sequence $\underline{\sigma} \in c_0(\N, \R^n)$ by setting $\sigma_0 := \sigma$ and $\sigma_t := 0$ when $t \in \N_*$. We consider the operator ${\mathcal E} : c_0(\N_*, \R^n) \rightarrow c_0(\N, \R^n)$ defined by ${\mathcal E} (\underline{x'}) := {\mathcal F} (\underline{x'}) + \underline{\sigma}$. Then ${\mathcal F}$ is linear continuous, ${\mathcal E}$ is affine continuous, and consequently these operators are continuously Fr\'echet differentiable, and for all $\underline{x'} \in  c_0(\N_*, \R^n)$ and $ \underline{\delta x'} \in  c_0(\N_*, \R^n)$, we have $D {\mathcal E} (\underline{x'}) \underline{\delta x'} = D {\mathcal F} (\underline{x'}) \underline{\delta x'} = {\mathcal F} ( \underline{\delta x'} )= (0, \delta x'_1, \delta x'_2, \cdot \cdot \cdot)$.
\end{remark}
\begin{proposition}\label{prop310}
Let $U$ be a nonempty closed subset of $\R^d$.
let $f : \R^n \times U \rightarrow \R^n$ be a mapping which satisfies the following properties.
\begin{itemize}
\item[({\bf a})] $f \in C^0( \R^n \times U, \R^n)$.
\item[({\bf b})] $f(0,u) = 0$  for all $u \in U$.
\item[({\bf c})] For all $u \in U$, for all $x \in \R^n$, $D_1f(x,u)$ exists and \\
$D_1f( \cdot , u) \in C^0( \R^n, {\mathfrak L}(\R^n, \R^n))$.
\item[({\bf d})] $D_1f$ transforms the nonempty bounded subsets of $\R^n \times U$ into bounded subsets of ${\mathfrak L}(\R^n, \R^n)$.
\item[({\bf e})] $\lim_{x \rightarrow 0}( \sup_{u \in B} \Vert D_1 f(x,u) \Vert_{\mathfrak L} ) = 0$ for all nonempty bounded subset $B \subset U$.
\end{itemize}
Then the operator  ${\mathcal T}(\underline{x}, \underline{u}) := (x_{t+1} -f(x_t,u_t))_{t \in \N}$ satisfies the following properties.
\begin{itemize}
\item[($\alpha$)] ${\mathcal T} \in C^0(c_0(\N, \R^n) \times {\ell}^{\infty}(\N, U), c_0(\N, \R^n))$
\item[($\beta$)] For all $( \underline{x}, \underline{u}) \in c_0(\N, \R^n) \times {\ell}^{\infty}(\N, U)$, $D_1 {\mathcal T}( \underline{x}, \underline{u})$ exists and for all 
$\underline{u} \in {\ell}^{\infty}(\N, U)$, $D_1 {\mathcal T}( \cdot , \underline{u}) \in C^0(c_0(\N, \R^n) , {\mathfrak L}(c_0(\N, \R^n), c_0(\N, \R^n)))$.
\end{itemize}
\end{proposition}
\begin{proof}
Let $B$ be a nonempty bounded subset of $U$. We fix $R \in (0, + \infty)$. For all $x \in \R^n$ such that $\Vert x  \Vert \leq R$, using (b), (c) and the mean value theorem we obtain 
$$\Vert f(x,u) \Vert \leq \Vert f(0,u) \Vert + \sup_{z \in [0,x]}  \Vert D_1f(z,u) \Vert_{\mathfrak L} \cdot \Vert x \Vert \leq   \sup_{z \in [0,x]} \Vert D_1f(z,u) \Vert_{\mathfrak L} \cdot \Vert x \Vert \Longrightarrow $$
$$\sup_{u \in B} \Vert f(x,u) \Vert \leq  \sup_{ \Vert z \Vert \leq R} \sup_{u \in B}  \Vert D_1f(z,u) \Vert_{\mathfrak L} \cdot \Vert x \Vert$$
which implies, using (d), the following property.
\begin{equation}\label{eq33}
\lim_{ x \rightarrow 0} (\sup_{u \in B} \Vert f(x,u) \Vert ) = 0.
\end{equation}
Therefore, from (a) and (\ref{eq33}) we obtain the conclusion ($\alpha$).

When we fix $\underline{u} \in {\ell}^{\infty}(\N, U)$, using Theorem A1.2 in \cite{BCr}, we obtain the conclusion ($\beta$).
\end{proof}
\vskip1mm
After the operators, we consider the criterion of Problem (P1).
\vskip2mm
\begin{proposition}\label{prop311}
Let $U$ be a nonempty closed subset of $\R^d$.
Let $\phi \in C^1(\R^n \times U, \R)$ and $\beta \in (0,1)$. We consider $J(\underline{x}, \underline{u}) := \sum_{ t=0}^{+ \infty} \beta^t \phi(x_t,u_t)$ when $\underline{x} \in c_0(\N, \R^n)$ and $\underline{u} \in {\ell}^{\infty}(\N, U)$. Then $J \in C^1(c_0(\N, \R^n) \times {\ell}^{\infty}(\N, U), \R)$ and for all $\underline{x} \in c_0(\N, \R^n)$, $\underline{u} \in {\ell}^{\infty}(\N, U)$, $ \underline{\delta x} \in c_0(\N, \R^n) $, $ \underline{\delta u} \in {\ell}^{\infty}(\N, \R^d)$, we have 
$$DJ(\underline{x}, \underline{u})( \underline{\delta x},  \underline{\delta u}) = \sum_{t = 0}^{\infty} \beta^t ( D_1 \phi(x_t,u_t) \delta x_t + D_2 \phi(x_t,u_t) \delta u_t).$$
\end{proposition}
\begin{proof}
We consider the Nemytskii operator 
$$N^1_{\phi} : {\ell}^{\infty}(\N, \R^n) \times {\ell}^{\infty} (\N, U) \rightarrow {\ell}^{\infty}(\N, \R), N^1_{\phi}(\underline{x}, \underline{u}) = (\phi(x_t,u_t))_{t \in \N}.$$
Using Theorem A1.2 in \cite{BCr}, we know that $N^1_{\phi} \in C^1({\ell}^{\infty}(\N, \R^n) \times {\ell}^{\infty} (\N, U), {\ell}^{\infty}(\N, \R))$ and for all $\underline{x} \in {\ell}^{\infty}(\N, \R^n)$, $\underline{u} \in {\ell}^{\infty}(\N, U)$, $ \underline{\delta x} \in {\ell}^{\infty}(\N, \R^n) $, $ \underline{\delta u} \in {\ell}^{\infty}(\N, \R^d)$, we have 
$$D N^1_{\phi}(\underline{x}, \underline{u})( \underline{\delta x},  \underline{\delta u}) = (D_1 \phi (x_t,u_t) \delta x_t + D_2 \phi (x_t,u_t) \delta u_t)_{ t \in \N}.$$
We also consider the other Nemytskii operator 
$$N_{\phi} : c_0(\N, \R^n) \times {\ell}^{\infty} (\N, U) \rightarrow {\ell}^{\infty}(\N, \R), N_{\phi}(\underline{x}, \underline{u}) = (\phi(x_t,u_t))_{t \in \N}.$$
Since $N_{\phi}$ is a restriction of $N^1_{\phi}$ we have $N_{\phi} \in C^1( c_0(\N, \R^n) \times {\ell}^{\infty} (\N, U),  {\ell}^{\infty}(\N, \R))$ 
and for all $\underline{x} \in c_0(\N, \R^n)$, $\underline{u} \in {\ell}^{\infty}(\N, U)$, $ \underline{\delta x} \in c_0(\N, \R^n) $, $ \underline{\delta u} \in {\ell}^{\infty}(\N, \R^d)$, we have 
$$D N_{\phi}(\underline{x}, \underline{u})( \underline{\delta x},  \underline{\delta u}) = (D_1 \phi (x_t,u_t) \delta x_t + D_2 \phi (x_t,u_t) \delta u_t)_{ t \in \N}.$$
Since $\beta \in (0,1)$, $(\beta^t)_{t \in \N} \in {\ell}^1(\N, \R)$ (the space of the absolutely convergent real series). We define the linear functional 
$$L(\underline{z}) := \sum_{t = 0}^{+ \infty} \beta^t z_t = \langle (\beta^t)_{t \in \N}, \underline{z} \rangle_{{\ell}^1, {\ell}^{\infty}}$$
where $\underline{z} \in {\ell}^{\infty}(\N, \R)$ and  $\langle \cdot , \cdot   \rangle_{{\ell}^1, {\ell}^{\infty}}$ denotes the duality bracket between ${\ell}^1(\N, \R)$ and ${\ell}^{\infty}(\N, \R)$. Using \cite{AB} (Theorem 15.22, p. 503), we know that $L$ is linear continuous on ${\ell}^{\infty}(\N, \R)$, and consequently we have $L \in C^1({\ell}^{\infty}(\N, \R), \R)$, and for all $\underline{z}$ and $  \underline{\delta z}$ in ${\ell}^{\infty}(\N, \R)$, we have $D L(\underline{z}) \underline{ \delta z} = L (\underline{\delta  z}) = \sum_{t = 0}^{+ \infty} (\beta^t  \cdot \delta z_t)$. Since $J = L \circ N_{\phi}$, $J$ is continuously differentiable as a composition of continuously differentiable mappings, and using the chain rule of the differential calculus, for all $\underline{x} \in c_0(\N, \R^n)$, $\underline{u} \in {\ell}^{\infty}(\N, U)$, $ \underline{\delta x} \in c_0(\N, \R^n) $, $\underline{\delta  u} \in {\ell}^{\infty}(\N, \R^d)$, we obtain 
\[
\begin{array}{rcl}
D J(\underline{x}, \underline{u})( \underline{\delta x},  \underline{\delta u}) & = & D L(N_{\phi}(\underline{x}, \underline{u})) D (N_{\phi}(\underline{x}, \underline{u})( \underline{\delta x},  \underline{\delta u})\\
\null & = & L(D (N_{\phi}(\underline{x}, \underline{u})( \underline{\delta x},  \underline{\delta u})\\
\null & = & \sum_{t=0}^{+ \infty} \beta^t (D_1 \phi (x_t,u_t) \delta x_t + D_2 \phi (x_t,u_t) \delta u_t).
\end{array}
\]
\end{proof}
Using similar arguments we establish the following result.
\vskip1mm
\begin{proposition}\label{prop312}
Let $U$ be a nonempty subset of $\R^d$,  $\beta \in (0,1)$ and $\phi \in C^0(\R^n \times U, \R)$ such that $D_1 \phi(x,u)$ exists for all $(x,u) \in \R^n \times U$ and, for all $u \in U$, $D_1 \phi( \cdot, u) \in C^0(\R^n, {\mathfrak L}(\R^n, \R^n))$.\\
Then $J \in C^0(c_0(\N, \R^n) \times {\ell}^{\infty}(\N, U), \R)$, and for all $\underline{x} \in c_0(\N, \R^n)$, for all $\underline{u} \in  {\ell}^{\infty}(\N, U)$, $D_1J(\underline{x}, \underline{u})$ exists and $D_1J(\cdot , \underline{u}) \in C^0(c_0(\N, \R^n), {\mathfrak L}(c_0(\N, \R^n), \R))$. Moreover, for all $\underline{x} \in c_0(\N, \R^n)$, for all $\underline{u} \in  {\ell}^{\infty}(\N, U)$, for all $\underline{\delta x} \in c_0(\N, \R^n)$, we have
$$D_1J(\underline{x}, \underline{u}) \underline{\delta x} = \sum_{t = 0}^{+ \infty} \beta^t D_1 \phi(x_t, u_t) \delta x_t.$$
\end{proposition}
\section{ Linear difference equations}
We establish a result on the existence of a solution of a nonhomogeneous linear equation which belongs to $c_0(\N_*, \R^n)$ when the second member belongs to $c_0(\N_*, \R^n)$. These results permit to obtain useful properties on the operator which represents the dynamical system of Problem (P1).
\vskip1mm
\begin{proposition}\label{prop41}
Let $(A_t)_{t \in \N_*}$ be a sequence in ${\mathfrak L}(\R^n, \R^n)$ and $e \in c_0(\N_*, \R^n)$. We consider the following Cauchy problem
\[
(DE)
\left\{
\begin{array}{rcl}
z_{t+1} & = & A_t z_t + e_t\\
z_1 & = & \zeta.
\end{array}
\right.
\]
We assume that $\sup_{t \in \N_*}  \Vert A_t \Vert_{\mathfrak L} < 1$.
Then the solution of (DE) belongs to $c_0(\N_*, \R^n)$.
\end{proposition}
\begin{proof}
We denote by $\underline{z}$ the solution of (DE). Doing a straightforward calculation we obtain, for all $t \in \N$, that 
$$z_{t+1} = (A_t \cdot \cdot \cdot A_1)\zeta + \sum_{i= 1}^{t-1} (A_t \cdot \cdot \cdot A_{i+1}) e_i + e_t.$$
Let $M > 0$ such that  $\sup_{t \in \N_*}  \Vert A_t \Vert_{\mathfrak L} \leq M < 1$. Therefore we have 
\[
\begin{array}{rcl}
\Vert z_{t+1} \Vert & \leq & (\prod_{i=1}^t \Vert A_i \Vert_{\mathfrak L}) \Vert \zeta \Vert + \sum_{i=1}^{t-1} (\prod_{j = i + 1}^{t-1} \Vert A_j \Vert_{\mathfrak L}) \Vert e_i \Vert + \Vert e_t \Vert \\
\null & \leq & M^{t} \Vert \zeta \Vert +  (\sum_{i=1}^{t-1} M^{t-i}) \Vert e \Vert_{\infty} + \Vert e \Vert_{\infty} \\
\null & = &  M^{t} \Vert \zeta \Vert + (\sum_{k= 0}^{t-1} M^k ) \Vert e \Vert_{\infty} \\
\null & \leq & \max \{ \Vert \zeta \Vert, \Vert e \Vert_{\infty} \} \sum_{k= 0}^{t} M^k \leq  \max \{ \Vert \zeta \Vert, \Vert e \Vert_{\infty} \}  \frac{1}{1-M} < + \infty
\end{array}
\]
which proves that $\underline{z} \in {\ell}^{\infty}(\N_*, \R^n)$.
\vskip1mm
From the definition of $\underline{z}$, using $ \limsup_{t \rightarrow + \infty} \Vert z_t \Vert  < + \infty$, we deduce
\[
\begin{array}{l}
\Vert z_{t+1} \Vert \leq \Vert A_t \Vert_{\mathfrak L} \cdot \Vert z_t \Vert + \Vert e_t \Vert \leq M \cdot \Vert z_t \Vert + \Vert e_t \Vert \Longrightarrow \\
\limsup_{t \rightarrow + \infty} \Vert z_t \Vert = \limsup_{t \rightarrow + \infty} \Vert z_{t+1} \Vert \leq M \cdot \limsup_{t \rightarrow + \infty} \Vert z_t \Vert + 0 \Longrightarrow \\
(1-M) \limsup_{t \rightarrow + \infty} \Vert z_t \Vert  \leq 0 \Longrightarrow  \limsup_{t \rightarrow + \infty} \Vert z_t \Vert = 0
\end{array}
\]
since $1-M > 0$, and therefore we obtain $ \lim_{t \rightarrow + \infty} z_t = 0$. 
\end{proof}
\begin{corollary}\label{cor42}
Let $(B_t)_{t \in \N_*}$ be a sequence in ${\mathfrak L}(\R^n, \R^n)$ and $d \in c_0(\N_*, \R^n)$. We consider the following Cauchy problem
\[
(DE1)
\left\{
\begin{array}{rcl}
w_{t+1} & = & B_t w_t + d_t\\
w_1 & = & \xi.
\end{array}
\right.
\]
We assume that there exists $t_* \in \N_*$ such that $\sup_{t \geq t_*} \Vert B_t \Vert_{\mathfrak L} < 1$.\\
Then the solution of (DE1) belongs to $c_0(\N_*, \R^n)$.
\end{corollary}
\begin{proof}
For all $t \in \N_*$, we set $A_t := B_{t + t_*}$ and $e_t := d_{t + t_*}$.\\
 Then we have $\sup_{ t \in \N} \Vert A_t \Vert_{\mathfrak L} < 1$.
We denote by $\underline{w}$ the solution of (DE1). We set $z_t := w_{t + t_*}$ for all $t \in \N$. Then we have $z_{t+1} = w_{t + 1 + t_*} = B_{t+t_*} w_{t+ t_*} + d_{t+t_*} = A_t z_t + e_t$ for all $t \in \N$ and $z_1 = w_{t_* +1}$. Using Proposition \ref{prop41} we obtain $\lim_{t \rightarrow + \infty} z_t = 0$, i.e. $\lim_{t \rightarrow + \infty} w_{t + t_*} = 0$ which implies $\lim_{t \rightarrow + \infty}w_t = 0$.
\end{proof}
\section{Static optimization}
In this section we establish a result in the form of a Karush-Kuhn-Tucker theorem in abstract Banach spaces, and we recall a result issued from the book of Ioffe and Tihomirov \cite{IT}. The first result is useful to prove our weak Pontryagin principles, and the second one is useful to prove our strong Pontryagin principles.
\begin{lemma}\label{lem51} Let ${\mathcal X}$, ${\mathcal V}$, ${\mathcal W}$ be real Banach spaces, and ${\mathcal U}$ be a nonempty subset of ${\mathcal V}$. Let ${\mathcal J} \in C^1({\mathcal X} \times {\mathcal U}, \R)$ and $\Gamma \in C^1({\mathcal X} \times {\mathcal U}, {\mathcal W})$. Let $(\hat{x}, \hat{u})$ be a solution of the following optimization problem
\[
\left\{
\begin{array}{rl}
{\rm Maximize} & {\mathcal J}(x,u)\\
{\rm when} & x \in {\mathcal X}, u \in {\mathcal U}, \Gamma(x,u) = 0.
\end{array}
\right.
\]
We assume that $D_1 \Gamma(\hat{x}, \hat{u})$ is invertible and that ${\mathcal U}$ is star-shaped with rerspect to $\hat{u}$. Then there exists $M \in {\mathcal W}^*$ which satisfies the following conditions.
\begin{itemize}
\item[(i)] $D_1 {\mathcal J}(\hat{x}, \hat{u}) + M \circ D_1 \Gamma(\hat{x}, \hat{u}) = 0$.
\item[(ii)] $\forall u \in {\mathcal U}$, $\langle D_2 {\mathcal J}(\hat{x}, \hat{u}) + M \circ D_2 \Gamma(\hat{x}, \hat{u}), u- \hat{u} \rangle \leq 0$.
\end{itemize}
\end{lemma}
\begin{proof}
Let ${\mathcal U}_1$ be an open subset of ${\mathcal V}$ such that ${\mathcal U} \subset {\mathcal U}_1$ and such that there exists $\Gamma_1 \in C^1({\mathcal X} \times {\mathcal U}_1, {\mathcal W})$ such that ${ \Gamma_1}_{\mid_{{\mathcal X} \times {\mathcal U}}} = \Gamma$. Since $D_1 \Gamma_1(\hat{x}, \hat{u}) = D_1 \Gamma (\hat{x}, \hat{u})$ is invertible, we can use the implicit function theorem and assert that there exist ${\mathcal N}_{\hat{x}}$ an open neighborhood of $\hat{x}$ in ${\mathcal X}$, ${\mathcal N}_{\hat{u}}$ an open convex neighborhood of $\hat{u}$ in ${\mathcal U}_1$, and a mapping $\pi \in C^1({\mathcal N}_{\hat{u}}, {\mathcal N}_{\hat{x}})$ such that 
$$\{ (x,u) \in {\mathcal N}_{\hat{x}} \times {\mathcal N}_{\hat{u}} : \Gamma_1(x,u) = 0 \} = \{ (\pi(u), u) : u \in  {\mathcal N}_{\hat{u}} \}.$$
Differentiating $\Gamma_1(\pi(u),u) = 0$ at $\hat{u}$ we obtain $D_1 \Gamma_1 (\hat{x}, \hat{u}) \circ D \pi(\hat{u}) + D_2 \Gamma_1 (\hat{x}, \hat{u}) = 0$ which implies
\begin{equation}\label{eq51}
D \pi (\hat{u}) = - (D_1 \Gamma (\hat{x}, \hat{u}))^{-1} \circ D_2 \Gamma (\hat{x}, \hat{u}).
\end{equation}
Since $(\hat{x}, \hat{u})$ is a solution of the initial problem, $\hat{u}$ is a solution of the following problem
\[
\left\{
\begin{array}{rl}
{\rm Maximize} & {\mathcal B}(u) \\
{\rm when} & u \in {\mathcal N}_{\hat{u}} \cap {\mathcal U}
\end{array}
\right.
\]
where ${\mathcal B}(u) = {\mathcal J}(\pi(u),u)$. Since ${\mathcal B}$ is differentiable (as a composition of differentiable mappings) and ${\mathcal N}_{\hat{u}} \cap {\mathcal U}$ is also star-shaped with respect to $\hat{u}$, a necessary condition of optimality for the last problem is
\begin{equation}\label{eq52}
\forall u \in {\mathcal N}_{\hat{u}} \cap {\mathcal U}, \langle D{\mathcal B}(\hat{u}), u - \hat{u} \rangle \leq 0
\end{equation}
since $0 \geq  \lim_{\theta \rightarrow 0+} \frac{1}{\theta}({\mathcal B}(\hat{u} + \theta (u - \hat{u})) - {\mathcal B}(\hat{u})) =   \langle D{\mathcal B}(\hat{u}), u - \hat{u} \rangle$.  When $u \in {\mathcal U}$, there exists $\theta_u \in (0,1)$ such that $(1- \theta_u) \hat{u}+ \theta_u u \in  {\mathcal N}_{\hat{u}} \cap {\mathcal U}$. Using (\ref{eq52}) we obtain
$$\theta_u \cdot  \langle D{\mathcal B}(\hat{u}), u - \hat{u} \rangle ) = \langle D{\mathcal B}(\hat{u}), \theta_u(u - \hat{u}) \rangle =  \langle D{\mathcal B}(\hat{u}), [(1 - \theta_u) \hat{u} + \theta_u u] - \hat{u} \rangle \leq 0,$$
and so we obtain
\begin{equation}\label{eq53}
\forall u \in {\mathcal U}, \langle D{\mathcal B}(\hat{u}), u - \hat{u} \rangle \leq 0.
\end{equation}
Using the chain rule we obtain
\begin{equation}\label{eq54}
D {\mathcal B}(\hat{u}) = D_1 {\mathcal J}(\hat{x}, \hat{u}) \circ D \pi( \hat{u}) + D_2 {\mathcal J}(\hat{x}, \hat{u}).
\end{equation}
We define 
\begin{equation}\label{eq55}
M := - D_1 {\mathcal J}(\hat{x}, \hat{u}) \circ (D_1 \Gamma (\hat{x}, \hat{u}))^{-1} \in {\mathcal W}^*.
\end{equation}
From (\ref{eq55}) we obtain
\begin{equation}\label{56}
 D_1 {\mathcal J}(\hat{x}, \hat{u}) + M \circ D_1 \Gamma (\hat{x}, \hat{u}) = 0.
 \end{equation}
 Using (\ref{eq54}) and (\ref{eq51}) we obtain $D {\mathcal B}(\hat{u}) = - D_1 {\mathcal J}(\hat{x}, \hat{u}) \circ (D_1 \Gamma (\hat{x}, \hat{u}))^{-1} \circ D_2 \Gamma (\hat{x}, \hat{u}) + D_2 {\mathcal J}(\hat{x}, \hat{u}) = M \circ D_2 \Gamma (\hat{x}, \hat{u}) + D_2 {\mathcal J}(\hat{x}, \hat{u})$, and therefore, from (\ref{eq53}) we obtain
\begin{equation}\label{57} 
\forall u \in {\mathcal U}, \langle  D_2 {\mathcal J}(\hat{x}, \hat{u}) + M \circ D_2 \Gamma (\hat{x}, \hat{u}) , u - \hat{u} \rangle \leq 0.
\end{equation}
\end{proof}
\begin{remark}\label{rem52}
There exist several results like this one in the books \cite{Co} and \cite{We} which use the convexity of $U$. In the necessary conditions of optimality we prefer to avoid the convexity of the sets; it is why we have established this lemma.

\end{remark}
As a corollary of the extremal principle in mixed problems (Theorem 3, p. 71 in \cite{IT}), we obtain the following result.
\begin{lemma}\label{lem53}
 Let ${\mathcal X}$, ${\mathcal V}$, ${\mathcal W}$ be real Banach spaces, and ${\mathcal U}$ be a nonempty subset of ${\mathcal V}$. Let ${\mathcal J} : {\mathcal X} \times {\mathcal U} \rightarrow \R$ and $\Gamma : {\mathcal X} \times {\mathcal U} \rightarrow {\mathcal W}$ be mappings. Let $(\hat{x}, \hat{u})$ be a solution of the following optimization problem
\[
\left\{
\begin{array}{rl}
{\rm Maximize} & {\mathcal J}(x,u)\\
{\rm when} & x \in {\mathcal X}, u \in {\mathcal U}, \Gamma(x,u) = 0.
\end{array}
\right.
\]
We assume that the following conditions are fulfilled.
\begin{itemize}
\item[(a)] For all $u \in {\mathcal U}$, $[x \mapsto \Gamma(x,u)]$ and $[x \mapsto {\mathcal J}(x,u)]$ are of class $C^1$ at $\hat{x}$.
\item[(b)] There exists a neighborhood ${\mathcal N}$ of $\hat{x}$ in ${\mathcal X}$ such that, for all $x \in {\mathcal N}$, for all $u', u'' \in {\mathcal U}$, for all $\theta \in [0,1]$, there exists $u \in {\mathcal U}$ which satisfies the following conditions
\[
\left\{
\begin{array}{ccl}
\Gamma(x,u) & =& (1- \theta) \Gamma(x,u') + \theta \Gamma(x, u'') \\
{\mathcal J}(x,u) & \geq & (1- \theta) {\mathcal J}(x,u') + \theta {\mathcal J}(x,u'').
\end{array}
\right.
\]
\item[(c)] The codimension of Im$D_1 \Gamma(\hat{x}, \hat{u})$ in ${\mathcal W}$ is finite.
\item[(d)] The set $\{ D_1  \Gamma(\hat{x}, \hat{u}) x +  \Gamma(\hat{x}, u) : x \in {\mathcal X}, u \in {\mathcal U} \}$ contains a neighborhood of the origine of ${\mathcal W}$.
\end{itemize}
Then there exists $M \in {\mathcal W}^*$ which satisfies the two following conditions.
\begin{itemize}
\item[(i)] $D_1 {\mathcal J}(\hat{x}, \hat{u}) + M \circ  D_1  \Gamma(\hat{x}, \hat{u}) = 0$.
\item[(ii)] For all $u \in {\mathcal U}$, $ {\mathcal J}(\hat{x}, \hat{u}) + M \Gamma(\hat{x}, \hat{u}) \geq  {\mathcal J}(\hat{x}, u) + M \Gamma(\hat{x}, u)$.
\end{itemize}
\end{lemma}
\section{Weak Pontryagin principle for (P1)}

We start by a translation of Problem (P1) into a more simple abstract optimization problem in Banach spaces.
We define the functional $J_1(\underline{x'}, \underline{u}) := J({\mathcal E}(\underline{x'}),  \underline{u})$ and the nonlinear operator ${\mathcal T}_1(\underline{x'}, \underline{u}) := {\mathcal T}({\mathcal E}(\underline{x'}),  \underline{u})$. Then we can translate (P1) into the following problem.
\[
(P2)
\left\{
\begin{array}{rl}
{\rm Maximize} & J_1(\underline{x'}, \underline{u})\\
{\rm when} & \underline{x'} \in c_0(\N_*, \R^n), \underline{u} \in {\ell}^{\infty}(\N, U)\\
\null & {\mathcal T}_1(\underline{x'}, \underline{u}) = \underline{0}.
\end{array}
\right.
\]
We consider the following list of assumptions.
\begin{itemize}
\item[(A1)] $U$ is a nonempty closed subset of $\R^d$.
\item[(A2)] $\phi \in C^1(\R^n \times U, \R)$ and $f \in C^1(\R^n \times U, \R^n)$.
\item[(A3)] There exists $u^0 \in U$ such that $f(0, u^0) = 0$ and $U$ is star-shaped with respect to $u^0$.
\item[(A4)] $\lim_{x \rightarrow 0}(\sup_{u \in B} \Vert Df(x,u) \Vert_{\mathfrak L}) = 0$ for all nonempty bounded subset $B \subset U$.
\end{itemize} 
\vskip1mm
\noindent
Recall that ${\ell}^1(\N, \R^{n*})$ can be assimilated to the dual topological space of $c_0(\N, \R^n)$, i.e. an element of ${\ell}^1(\N, \R^{n*})$ can be considered as a continuous linear functional on $c_0(\N, \R^n)$, \cite{AB} (Theorem 15.9, p. 498).
\vskip1mm
\begin{lemma}\label{lem61}
We assume (A1-A4) fulfilled. Let $(\hat{\underline{x'}}, \hat{\underline{u}})$ be a solution of (P2). \\
Then there exists $\underline{q} \in {\ell}^1(\N, \R^{n*})$ which satisfies the two following conditions.
\begin{itemize}
\item[(i)] $D_1J_1(\hat{\underline{x'}}, \hat{\underline{u}}) + \underline{q} \circ D_1 {\mathcal T}_1 (\hat{\underline{x'}}, \hat{\underline{u}}) = \underline{0}$.
\item[(ii)] For all $\underline{u} \in {\ell}^{\infty}(\N, U)$, $\langle D_2J_1(\hat{\underline{x'}}, \hat{\underline{u}}) + \underline{q} \circ D_2 {\mathcal T}_1 (\hat{\underline{x'}}, \hat{\underline{u}}), \underline{u} - \hat{\underline{u}} \rangle \leq 0$.
\end{itemize}
\end{lemma}
\begin{proof}
Using Remark \ref{rem39} and Proposition \ref{prop311}, $J_1$ is of class $C^1$ as a composition of mappings of class $C^1$. Using Remark \ref{rem39} and Corollary \ref{cor38}, ${\mathcal T}_1$ is of class $C^1$ as a composition of operators of class $C^1$.
\vskip1mm
We set $\hat{B} := \{ \hat{u}_t : t \in \N \}$. Then $\hat{B}$ is nonempty bounded in $U$ since $\hat{\underline u} \in {\ell}^{\infty}(\N, U)$. For all $t \in \N_*$, we have
$$\Vert D_1 f(\hat{x'}_t, \hat{u}_t) \Vert_{\mathfrak L} \leq \Vert D f(\hat{x'}_t, \hat{u}_t) \Vert_{\mathfrak L}  \leq \sup_{u \in \hat{B}} \Vert D f(\hat{x}'_t,u) \Vert_{\mathfrak L}$$
and therefore, using (A4), we obtain $\lim_{ t \rightarrow + \infty} \Vert D_1 f(\hat{x'}_t, \hat{u}_t) \Vert_{\mathfrak L} = 0$. Therefore there exists $t_* \in \N$ such that $\sup_{ t \geq t_*} \Vert D_1 f(\hat{x'}_t, \hat{u}_t) \Vert_{\mathfrak L} < 1$.
Note that to solve equation (DE1) of Section 4, with $B_t := D_1f(\hat{x'}_t, \hat{u}_t)$ when $t \in \N_*$, is equivalent to solve the equation $D_1 {\mathcal T}_1(\hat{\underline{x'}}, \hat{\underline{u}}) \underline{\delta x'} = \underline{e}$ where $\underline{e} \in c_0(\N, \R^n)$ and the unknown variable is $\underline{\delta x'} \in c_0(\N_*, \R^n)$. We can use Corollary \ref{cor42} and assert that $D_1 {\mathcal T}_1(\hat{\underline{x'}}, \hat{\underline{u}})$ is surjective and it is clearly injective, and consequently $D {\mathcal T}_1(\hat{\underline{x'}}, \hat{\underline{u}})$ is also invertible. Therefore we can use Lemma \ref{lem51} and assert that there exists a Lagrange multiplier $\underline{q} \in c_0(\N_ù, \R^n)^* = {\ell}^1(\N_ù, \R^{n*})$ which satisfies the announced conclusions.
\end{proof}
\begin{theorem}\label{th62}
We assume (A1-A4) fulfilled. Let $(\hat{\underline{x}}, \hat{\underline{u}})$ be a solution of (P1).
Then there exists $\underline{p} \in {\ell}^1(\N_*, \R^{n*})$ such that the following relations hold.
\vskip1mm
\noindent
{\bf (AE1)} \hskip3mm $p_t = p_{t+1} \circ D_1f(\hat{x}_t, \hat{u}_t) + D_1 \phi(\hat{x}_t, \hat{u}_t) $ for all $t \in \N_*$\\
{\bf (WM1)} \hskip3mm $\langle p_{t+1} \circ D_2f(\hat{x}_t, \hat{u}_t) + D_2 \phi(\hat{x}_t, \hat{u}_t) , u - \hat{u}_t \rangle \leq 0$ for all $u \in U$, for all $t \in \N$.
\end{theorem}
\begin{proof}
We define $\hat{\underline{x'}}$ by setting $\hat{x'}_t := \hat{x}_t$ when $t \in \N_*$. Since $(\hat{\underline{x}}, \hat{\underline{u}})$ is a solution of (P1), $(\hat{\underline{x'}},  \hat{\underline{u}})$ is a solution of (P2). Then Lemma \ref{lem51} provides $\underline{q} \in {\ell}^1(\N, \R^{n*})$ such that 
\begin{equation}\label{eq61}
\left.
\begin{array}{rcl}
D_1J_1(\hat{\underline{x'}}, \hat{\underline{u}}) + \underline{q} \circ D_1{\mathcal T}_1 (\hat{\underline{x'}}, \hat{\underline{u}}) &=& 0\\
\langle D_2J_1(\hat{\underline{x'}}, \hat{\underline{u}}) + \underline{q} \circ D_2{\mathcal T}_1 (\hat{\underline{x'}}, \hat{\underline{u}}), \underline{u} - \hat{\underline{u}} \rangle &\leq& 0
\end{array}
\right\}
\end{equation}
for all $ \underline{u} \in {\ell}^{\infty}(\N, U)$.
Now we translate these conditions to obtain the conclusions of our theorem.
Using Remark \ref{rem39}, Proposition \ref{prop311} and the chain rule we obtain
\[
\begin{array}{rcl}
D_1 J_1(\hat{\underline{x'}}, \hat{\underline{u}}) \underline{\delta x'} &= & D_1 J( {\mathcal E}(\hat{\underline{x'}}),\hat{\underline{u}}) D {\mathcal E}(\hat{\underline{x'}})\underline{\delta x'} = D_1 J( \hat{\underline{x}},\hat{\underline{u}}) {\mathcal F}(\underline{\delta x'})\\
\null &=& \beta^0 D_1 \phi(\sigma, \hat{u}_0) 0 + \sum_{t=1}^{+\infty} \beta^t D_1 \phi(\hat{x}_t, \hat{u}_t) \delta x_t,
\end{array}
\]
and therefore we have
\begin{equation}\label{eq62}
D_1 J_1(\hat{\underline{x'}}, \hat{\underline{u}}) \underline{\delta x'} = \sum_{t=1}^{+\infty} \beta^t D_1 \phi(\hat{x}_t, \hat{u}_t) \delta x'_t.
\end{equation}
Using the same arguments we have
$D_2 J_1(\hat{\underline{x'}}, \hat{\underline{u}}) \underline{\delta u} = D_2 J ( {\mathcal E}(\hat{\underline{x'}}),\hat{\underline{u}}) \underline{\delta u}$
which implies 
\begin{equation}\label{eq63}
 D_2 J_1(\hat{\underline{x'}}, \hat{\underline{u}}) \underline{\delta u} = \sum_{t = 0}^{+\infty} \beta^t D_2 \phi (\hat{x}_t, \hat{u}_t) \delta u_t.
\end{equation}
Using Corollary \ref{cor38} and Remark \ref{rem39} and the chain rule we obtain 
\[
\begin{array}{rcl}
D_1 {\mathcal T}_1(\hat{\underline{x'}}, \hat{\underline{u}}) \underline{\delta x'} &=& D_1{\mathcal T}({\mathcal E}(\hat{\underline{x'}}), \hat{\underline{u}})D{\mathcal E}(\hat{\underline{x'}})\underline{\delta x'}\\
\null & = & D_1{\mathcal T}(\hat{\underline{x}}, \hat{\underline{u}}) {\mathcal F}(\underline{\delta x'})\\
\null & = & (\delta x_1 - D_1 f(\sigma, \hat{u}_0)0, (\delta x_{t+1} - D_1f( \hat{x}_t, \hat{u}_t) \delta x_t)_{ t \in \N_*}),
\end{array}
\]
and therefore we have
\begin{equation}\label{eq64}
D_1 {\mathcal T}_1(\hat{\underline{x'}}, \hat{\underline{u}}) \underline{\delta x'} = (\delta x'_1,  (\delta x'_{t+1} - D_1f( \hat{x}_t, \hat{u}_t) \delta x'_t)_{ t \in \N_*}).
\end{equation}
Using the same arguments, we obtain $D_2 {\mathcal T}_1(\hat{\underline{x'}}, \hat{\underline{u}}) \underline{\delta u} = D_2 {\mathcal T}({\mathcal E}(\hat{\underline{x'}}), \hat{\underline{u}}) \underline{\delta u} = $\\
$ D_2 {\mathcal T}(\hat{\underline{x}}, \hat{\underline{u}}) \underline{\delta u}$ which implies
\begin{equation}\label{eq65}
D_2 {\mathcal T}_1(\hat{\underline{x'}}, \hat{\underline{u}}) \underline{\delta u} = (-D_2 f(\hat{x}_t, \hat{u}_t) \delta u_t)_{ t \in \N}.
\end{equation}
Using (\ref{eq61}) we calculate $\underline{q} \circ D_1 {\mathcal T}_1 (\hat{\underline{x'}}, \hat{\underline{u}}) \underline{\delta x'} = \langle q_0, \delta x'_1 \rangle +
\sum_{t=1}^{+ \infty} \langle q_t, \delta x'_{t+1} \rangle - \sum_{t=1}^{+ \infty} q_t \circ D_1 f(\hat{x}_t, \hat{u}_t) \delta x'_t =  \sum_{t=0}^{+ \infty} \langle q_t, \delta x'_{t+1} \rangle - \sum_{t=1}^{+ \infty} q_t \circ D_1 f(\hat{x}_t, \hat{u}_t) \delta x'_t =
\sum_{t=1}^{+\infty} \langle q_{t-1}, \delta x'_t  \rangle - \sum_{t=1}^{+ \infty} \langle q_t \circ D_1 f(\hat{x}_t, \hat{u}_t), \delta x'_t \rangle$ which implies
\begin{equation}\label{eq66}
\underline{q} \circ D_1 {\mathcal T}_1 (\hat{\underline{x'}}, \hat{\underline{u}}) \underline{\delta x'} =
\sum_{t=1}^{+\infty} \langle q_{t-1} -q_t \circ D_1 f(\hat{x}_t, \hat{u}_t), \delta x'_t  \rangle.
\end{equation}
Using (\ref{eq61}), (\ref{eq62}) and (\ref{eq66}) we obtain
\begin{equation}\label{eq67}
\sum_{t=1}^{+\infty} \beta^t D_1 \phi(\hat{x}_t, \hat{u}_t) \delta x'_t = \sum_{t=1}^{+\infty} \langle q_{t-1} -q_t \circ D_1 f(\hat{x}_t, \hat{u}_t), \delta x'_t  \rangle.
\end{equation}
We fix $t \in \N_*$, we set $\delta x'_s = 0$ when $s \neq t$ and $\delta x'_t$ varies in $\R^n$, then from the last equation we obtain $\beta^t D_1 \phi(\hat{x}_t, \hat{u}_t) = q_{t-1} -q_t \circ D_1 f(\hat{x}_t, \hat{u}_t)$, which implies, for all $t \in \N_*$,
\begin{equation}\label{eq68}
q_{t-1} = q_t \circ D_1 f(\hat{x}_t, \hat{u}_t) + \beta^t D_1 \phi(\hat{x}_t, \hat{u}_t).
\end{equation}
We define $\underline{p} \in {\ell}^1(\N_*, \R^{n*})$ by setting $p_t := q_{t-1}$. Then (\ref{eq68}) implies (AE1).
\vskip1mm
From (\ref{eq65}) we obtain 
$$\sum_{t=0}^{+ \infty} \langle  q_t \circ D_2 f(\hat{x}_t, \hat{u}_t),  u_t - \hat{u}_t \rangle + \langle \underline{q} \circ D_2 {\mathcal T}_1(\hat{\underline{x'}}, \hat{\underline{u}}) \underline{u} - \hat{\underline{u}} \rangle \leq 0 $$
for all $\underline{u} \in {\ell}^{\infty}(\N, U)$, therefore from (\ref{eq61}) and (\ref{eq63}) we obtain 
$$\sum_{t=0}^{+ \infty} \langle  q_t \circ D_2 f(\hat{x}_t, \hat{u}_t),  u_t - \hat{u}_t \rangle + \sum_{t = 0}^{+\infty} \beta^t \langle D_2 \phi (\hat{x}_t, \hat{u}_t),  u_t - \hat{u}_t \rangle \leq 0$$
for all $\underline{u} \in {\ell}^{\infty}(\N, U)$. We fix $t \in \N$, we take $u_s = \hat{u}_s$ when $s \neq t$, and $u_t$ varies in $U$. Then we obtain $\langle  q_t \circ D_2 f(\hat{x}_t, \hat{u}_t) + \beta^t D_2 \phi (\hat{x}_t, \hat{u}_t),  u_t - \hat{u}_t \rangle \leq 0$ which implies, for all $t \in \N$ and for all $u_t \in U$ 
\begin{equation}\label{eq69}
\langle  q_t \circ D_2 f(\hat{x}_t, \hat{u}_t) + \beta^t D_2 \phi (\hat{x}_t, \hat{u}_t),  u_t - \hat{u}_t \rangle \leq 0.
\end{equation}
Replacing $q_t$ by $p_{t+1}$ in this last equation we obtain (WM1). 
\end{proof}
\begin{remark}\label{rem63}
In Theorem \ref{th62}, (AE1) means Adjoint Equation for (P1), (WM1) means Weak Maximum principle for (P1). Since $\underline{p} \in {\ell}^1(\N_*, \R^{n*})$, note that and the transversality condition at infinity for problem (P1), $\lim_{t \rightarrow + \infty} p_t = 0$, is satisfied.
\end{remark}
\section{Weak Pontryagin principle for (P)}
In this section we translate the main result of Section 6 on (P1) into a result on (P).\\
We introduce the following conditions
\begin{itemize}
\item[(B1)] $U$ is a nonempty closed subset of $\R^d$.
\item[(B2)] $\psi \in C^1(\R^n \times U, \R)$ and $g \in C^1(\R^n \times U, \R^n)$.
\item[(B3)] There exists $u^0 \in U$ such that $g(y_{\infty},u^0) = y_{\infty}$ and $U$ is star-shaped with respect to $u^0$.
\item[(B4)] $\lim_{ y \rightarrow y_{\infty}} ( \sup_{u \in B} \Vert Dg(y,u) \Vert)  = 0$ for all nonempty bounded subset $B \subset U$.
\end{itemize}
\begin{theorem}\label{th71}
We assume (B1-B4) fulfilled. let $(\hat{\underline y}, \hat{\underline u})$ be a solution of Problem (P). 
Then there exists $\underline{p} \in {\ell}^1(\N_*, \R^{n*})$ such that the following relations hold.
\vskip1mm
\noindent
{\bf (AE)} \hskip3mm $p_t = p_{t+1} \circ D_1 g(\hat{y}_t, \hat{u}_t) + \beta^t D_1 \psi(\hat{y}_t, \hat{u}_t)$ for all $t \in \N_*$.\\
{\bf (WM)} \hskip3mm $\langle p_{t+1} \circ D_2 g(\hat{y}_t, \hat{u}_t) + \beta^t D_2 \psi(\hat{y}_t, \hat{u}_t) , u - \hat{u} \rangle \leq 0$ for all $u \in U$, for all $t \in \N$
\end{theorem}
\begin{proof}
Using Section 2, since $(\hat{\underline y}, \hat{\underline u})$ is a solution of (P), $(\hat{\underline x}, \hat{\underline u})$ is a solution of (P1) with $\hat{x}_t = \hat{y}_t - y_{\infty}$. For all $j \in \{ 1,2,3,4 \}$, (Bj) implies (Aj) and so the conclusions of Theorem \ref{th62} hold. We conserve the same $\underline{p}$, and we translate to see that (AE1) implies (AE) and (WM1) implies (WM).
\end{proof}
\section{Strong Pontryagin principle for (P1)}
First we introduce the Hamiltonian of Pontryagin which is defined, for all $t \in \N$, as follows
$$H_t : \R^n \times U \times \R^{n*} \rightarrow \R, H_t(x,u,p) := \beta^t \phi(x,u) + \langle p, f(x,u) \rangle.$$
Note that the condition (WM1) of Theorem \ref{th82} is equivalent to the condition 
$$\langle D_2 H_t(\hat{x}_t, \hat{u}_t, p_{t+1}), u_t - \hat{u}_t \rangle \leq 0$$
for all $u \in U$ and for all $t \in \N$. In this section we want replace (WM1) by the strengthened condition $ H_t(\hat{x}_t, \hat{u}_t, p_{t+1}) = \max_{u \in U}  H_t(\hat{x}_t, u, p_{t+1})$ for all $t \in \N$. Note that (WM1) can be viewed as a first-order necessary condition of the optimality of $ H_t(\hat{x}_t, \cdot , p_{t+1})$ at $\hat{u}_t$ on $U$.
\vskip1mm
We consider the following conditions
\begin{itemize}
\item[(C1)] $U$ is a nonempty compact subset of $\R^d$.
\item[(C2)]  $\phi \in C^0(\R^n \times U, \R)$ and $f \in C^0(\R^n \times U, \R^n)$.
\item[(C3)]  For all $u \in U$, $f(0,u) = 0$.
\item[(C4)]  For all $u \in U$,  $D_1f(x,u)$ and  $D_1\phi (x,u)$ exist for all $x \in \R^n$, and $D_1f( \cdot, u) \in C^0(\R^n, {\mathfrak L}(\R^n, \R^n))$, and $D_1 \phi( \cdot, u) \in C^0(\R^n,\R^{n*})$.
\item[(C5)] $D_1f$ transforms the nonempty bounded subsets of $\R^n \times U$ into bounded subsets of ${\mathfrak L}(\R^n, \R^n)$.

\item[(C6)] For all nonempty bounded subset $B \subset U$, $\lim_{x \rightarrow 0}( \sup_{u \in B} \Vert D_1 f(x,u) \Vert_{\mathfrak L}) = 0$.
\item[(C7)] For all $t \in \N$, for all $x_t \in \R^n$, for all $u'_t$, $u''_t \in U$ and for all $\theta \in (0,1)$, there exists $u_t \in U$ such that
\[
\left\{
\begin{array}{rcl}
\phi(x_t, u_t) & \geq & (1- \theta) \phi(x_t, u'_t) + \theta \phi (x_t, u''_t)\\
f(x_t, u_t) & = & (1- \theta) f(x_t, u'_t) + \theta f (x_t, u''_t).
\end{array}
\right.
\]
\end{itemize}
\begin{lemma}\label{lem81}
Under the assumptions (C1-C7) let $(\hat{\underline{x'}}, \hat{\underline{u}})$ be a solution of Problem (P2) defined in Section 5. Then there exists $\underline{q} \in {\ell}^1(\N, \R^{n*})$ which satisfies the following properties.
\begin{itemize}
\item[(1)] $D_1 J_1(\hat{\underline{x'}}, \hat{\underline{u}}) + \underline{q} D_1 {\mathcal T}_1(\hat{\underline{x'}}, \hat{\underline{u}}) = \underline{0}$.
\item[(2)] $ J_1(\hat{\underline{x'}}, \hat{\underline{u}}) + \langle \underline{q}, N'_f(\hat{\underline{x'}}, \hat{\underline{u}}) \rangle_{c_0, {\ell}^1} = \max_{\underline{u} \in {\ell}^{\infty}(\N, U)} ( J_1(\hat{\underline{x'}}, \underline{u}) + \langle \underline{q}, N'_f(\hat{\underline{x'}}, \underline{u}) \rangle_{c_0, {\ell}^1})$.
\end{itemize}
\end{lemma}
\begin{proof}
We want to use Lemma \ref{lem53} with ${\mathcal J} = J_1$, $\Gamma = {\mathcal T}_1$.\\
Since (C1-C6) imply that $U$ is closed and that the conditions (a, b, c, d, e) of Proposition \ref{prop310} hold, we obtain that ${\mathcal T}$ and $D_1 {\mathcal T}( \cdot, \underline{u})$ are continuous, and using Remark \ref{rem39} we obtain that ${\mathcal T}_1$ and $D_1 {\mathcal T}_1( \cdot, \underline{u})$ are continuous. Using Proposition \ref{prop312}, from (C2) and (C4) we obtain that $J$ and $D_1J( \cdot, \underline{u})$ are continuous, and using Remark \ref{rem39} we obtain that $J_1$ and $D_1J_1( \cdot, \underline{u})$ are continuous. And so the assumption (a) of Lemma \ref{lem53} is fulfilled.\\
Since $U$ is bounded, from (C7) we obtain assumption (b) of Lemma \ref{lem53}.\\
Proceeding as in the proof of Lemma \ref{lem61}, from (C6), with $B_t := D_1f(\hat{x}_t, \hat{u}_t)$, we obtain the assumptions of Corollary \ref{cor42} which implies that\\
 $D_1 {\mathcal T}_1(\hat{\underline{x'}}, \hat{\underline{u}})$ is surjective from $c_0(\N_*, \R^n)$ onto $c_0(\N, \R^n)$, and since it is clearly injective, it is invertible. Using the Isomorphism Theorem of Banach, this invertibility implies the assumptions (c) and (d) of Lemma \ref{lem53}. \\
Consequently we can use Lemma \ref{lem53} and we obtain the conclusions with $\underline{q} = M$.   
\end{proof}
\begin{theorem}\label{th82}
Under the assumptions (C1-C7), let $(\hat{\underline{x}}, \hat{\underline{u}})$ be a solution of Problem (P) defined in Section 5. Then there exists $\underline{p} \in {\ell}^1(\N_*, \R^{n*})$ which satisfies the following properties.
\vskip1mm
\noindent
{\bf (AE1)} \hskip3mm $D_1 \phi(\hat{x}_t, \hat{u}_t) + p_{t+1} \circ D_1 f(\hat{x}_t, \hat{u}_t) = 0$ for all $t \in \N_*$.\\
{\bf (MP1)} \hskip3mm  $ \phi(\hat{x}_t, \hat{u}_t) + \langle p_{t+1},f(\hat{x}_t, \hat{u}_t) \rangle = \max_{u \in  U} ( \phi(\hat{x}_t, u) + \langle p_{t+1}, f(\hat{x}_t,u) \rangle )$  for all $t \in \N$.
\end{theorem}
\begin{proof}
Proceeding as in the proof of Theorem \ref{th62}, conclusion (1) of Lemma \ref{lem81} implies (AE1). A straightforward translation of conclusion (2) of Lemma \ref{lem81} provides (MP1).
\end{proof}
\section{Strong Pontryagin principle for (P)}
In this section we translate the strong Pontryagin principle on (P1) into a result on (P).\\
We consider the following conditions.
\begin{itemize}
\item[(D1)] $U$ is a nonempty compact subset of $\R^d$.
\item[(D2)] $\psi \in C^0(\R^n \times U, \R)$ and $g \in C^0(\R^n \times U, \R^n)$.
\item[(D3)] For all $u \in U$, $g(y_{\infty}, u) = y_{\infty}$.
\item[(D4)] For all $(y,u) \in \R^n \times U$, $D_1 \psi (y,u)$ and $D_1g(y,u)$ exist and, for all $u \in U$, $D_1 \psi ( \cdot, u) \in C^0(\R^n, \R^{n*})$, $D_1 g( \cdot , u) \in C^0( \R^n, {\mathfrak L}(\R^n, \R^n))$.
\item[(D5)] $D_1 g$ transforms the nonempty bounded subsets of $\R^n \times U$ in bounded subsets of ${\mathfrak L}(\R^n, \R^n)$.
\item[(D6)] $\lim_{y \rightarrow y_{\infty}} (\sup_{u \in B} \Vert D_1 g(y,u) \Vert_{\mathfrak L} ) = 0$ for all nonempty bounded subset $B$ of $U$.
\item[(D7)] For all $t \in \N$, for all $y_t \in \R^n$, for all $u'_t$, $u''_t \in U$ and for all $\theta \in (0,1)$, there exists $u_t \in U$ such that
\[
\left\{
\begin{array}{rcl}
\psi(y_t, u_t) & \geq & (1- \theta) \psi(y_t, u'_t) + \theta \psi (y_t, u''_t)\\
g(y_t, u_t) & = & (1- \theta) g(y_t, u'_t) + \theta g (y_t, u''_t).
\end{array}
\right.
\]
\end{itemize}
\begin{theorem}\label{th91}
Under the assumptions (D1-D7) let $(\hat{\underline{y}}, \hat{\underline{u}})$ be a solution of Problem (P).  Then there exists $\underline{p} \in {\ell}^1(\N_*, \R^{n*})$ which satisfies the following properties.
\vskip1mm
\noindent
{\bf (AE)} $D_1 \psi(\hat{y}_t, \hat{u}_t) + p_{t+1} \circ D_1 g(\hat{y}_t, \hat{u}_t) = 0$ for all $t \in \N_*$.\\
{\bf (MP)} $ \psi(\hat{y}_t, \hat{u}_t) + \langle p_{t+1}, g(\hat{y}_t, \hat{u}_t) \rangle = \max_{u \in  U} ( \psi(\hat{y}_t, u) + \langle p_{t+1}, g(\hat{y}_t,u) \rangle )$  for all $t \in \N$.
\end{theorem}
\begin{proof}
Using Section 2, since $(\hat{\underline{y}}, \hat{\underline{u}})$ is a solution of Problem (P),  $(\hat{\underline{x}}, \hat{\underline{u}})$ is a solution of Problem (P1). For all $j \in \{ 1,...,7 \}$ note that (Cj) implies (Dj). Therefore the assumptions of Theorem \ref{th82} are fulfilled, and so its conclusions hold. Using Section 2, we conserve $\hat{\underline{u}}$ and $\underline{p}$, we set $\hat{x}_t = \hat{y}_t - y_{\infty}$ for all $t \in \N$, and the translation of (AE1) gives (AE) and the translation of (MP1) gives (MP).
\end{proof}
\section{Sufficient conditions for (P1)}
In this section we establish a resullt of sufficient condition of optimality which uses the adjoint equation and the weak maximum principle and the concavity of the Hamiltonian with respect the state variable and the control variable.
\begin{theorem}\label{th101}
Let $U$ be a nonempty convex subset of $\R^d$, $\beta \in (0,1)$, $\sigma \in \R^n$ and two mappings $\phi : \R^n \times U \rightarrow \R$ and $f : \R^n \times U \rightarrow \R^n$. \\
Let $(\hat{\underline{x}}, \hat{\underline{u}}) \in c_0(\N, \R^n) \times  {\ell}^{\infty}(\N, U)$ and $\underline{p} \in {\ell}^1(\N_*, \R^{n*})$. Assume that the following conditions hold.
\begin{itemize}
\item[(i)] $\hat{x}_{t+1} = f(\hat{x}_t, \hat{u}_t)$ for all $t \in \N$, and $\hat{x}_0 = \sigma$.
\item[(ii)] $\phi \in C^1( \R^n \times U , \R)$ and $f \in C^1(\R^n \times U , \R^n)$.
\item[(iii)]  $\phi$ transforms bounded subsets of $\R^n \times U$ into bounded subsets of $\R$. 
\item[(iv)]  $p_t = p_{t+1} \circ D_1f(\hat{x}_t, \hat{u}_t) + \beta^t D_1 \phi(\hat{x}_t, \hat{u}_t)$ for all $t \in \N_*$.
\item[(v)] $\langle p_{t+1} \circ D_2f(\hat{x}_t, \hat{u}_t) + D_2 \phi(\hat{x}_t, \hat{u}_t) , u - \hat{u}_t \rangle \leq 0$ for all $u \in U$, for all $t \in \N$. 
\item[(vi)] The function $[(x,u) \mapsto \langle p_{t+1}, f(x,u) \rangle +  \beta^t \phi(x,u))]$ is concave on $\R^n \times U$ for all $t \in \N$.
\end{itemize}
Then $(\hat{\underline{x}}, \hat{\underline{u}})$ is a solution of (P1).
\end{theorem}
\begin{proof}
Let $(\underline{x}, \underline{u})$ be an admissible process for (P1), i.e. $\underline{x} \in c_0(\N, \R^n)$, $\underline{u} \in {\ell}^{\infty}(\N, U)$, $x_{t+1} = f(x_t,u_t)$ for all $t \in \N$, and $x_0 = \sigma$. From (iii), since $\{ \phi(x_t,u_t) : t \in \N \}$ is bounded, $J(\underline{x}, \underline{u}) = \sum_{t = 0}^{+\infty} \beta^t \phi(x_t,u_t)$ exists in $\R$. 
From (ii) and (iv) we obtain
\begin{equation}\label{eq10.1}
D_1 H_t(\hat{x}_t, \hat{u}_t, p_{t+1}) = p_t.
\end{equation}
From (vi) we obtain, for all $t \in \N$,
\begin{equation}\label{eq10.3}
\left.
\begin{array}{r}
H_t((\hat{x}_t, \hat{u}_t, p_{t+1}) - H_t(x_t,u_t,p_{t+1}) \\
- \langle D_1 H_t(\hat{x}_t, \hat{u}_t, p_{t+1}), \hat{x}_t - x_t \rangle - \langle D_2 H_t(\hat{x}_t, \hat{u}_t, p_{t+1}), \hat{u}_t - u_t \rangle \geq 0.
\end{array}
\right\}
\end{equation}
From (v) the following relation holds for all $t \in \N$
\begin{equation}\label{eq10.4}
\langle D_2 H_t(\hat{x}_t, \hat{u}_t, p_{t+1}), \hat{u}_t - u_t \rangle \geq 0.
\end{equation}
For all $t \in \N$ we have
\[
\begin{array}{rcl}
\beta^t \phi (\hat{x}_t, \hat{u}_t) - \beta^t \phi(x_t,u_t) &=&  H_t( \hat{x}_t, \hat{u}_t, p_{t+1}) - \langle p_{t+1}, f( \hat{x}_t, \hat{u}_t) \rangle \\
\null & \null & -H_t(x_t,u_t, p_{t+1}) + \langle p_{t+1}, f(x_t,u_t)\\
\null& =& H_t( \hat{x}_t, \hat{u}_t, p_{t+1}) - H_t(x_t,u_t, p_{t+1}) \\
\null & \null & -\langle p_{t+1}, \hat{x}_{t+1} - x_{t+1} \rangle.
\end{array}
\]
Then, using (\ref{eq10.1}) and (\ref{eq10.4}) we obtain
\[
\begin{array}{rcl}
\beta^t \phi (\hat{x}_t, \hat{u}_t) - \beta^t \phi(x_t,u_t) &\geq & H_t( \hat{x}_t, \hat{u}_t, p_{t+1}) -H_t(x_t,u_t, p_{t+1}) \\
\null & \null & - \langle D_2H_t( \hat{x}_t, \hat{u}_t, p_{t+1}), \hat{u}_t - u_t \rangle\\
\null & \null & - \langle D_1H_{t+1}( \hat{x}_{t+1}, \hat{u}_{t+1}, p_{t+2}), \hat{x}_{t+1} - x_{t+1} \rangle  
\end{array}
\]
which implies
\[
\begin{array}{rcl}
\beta^t \phi (\hat{x}_t, \hat{u}_t) - \beta^t \phi(x_t,u_t) &\geq & [H_t( \hat{x}_t, \hat{u}_t, p_{t+1}) -H_t(x_t,u_t, p_{t+1}) \\
\null & \null & - \langle D_1H_t( \hat{x}_t, \hat{u}_t, p_{t+1}), \hat{x}_t - x_t \rangle\\
\null & \null & - \langle D_2H_t( \hat{x}_t, \hat{u}_t, p_{t+1}), \hat{u}_t - u_t \rangle ] \\
\null & \null & + [  \langle D_1H_t( \hat{x}_t, \hat{u}_t, p_{t+1}), \hat{x}_t - x_t \rangle\\
\null & \null & - \langle D_1H_{t+1}( \hat{x}_{t+1}, \hat{u}_{t+1}, p_{t+2}), \hat{x}_{t+1} - x_{t+1} \rangle ]
\end{array}
\]
and using (\ref{eq10.3}) we obtain
\[
\begin{array}{rcl}
\beta^t \phi (\hat{x}_t, \hat{u}_t) - \beta^t \phi(x_t,u_t) &\geq &[  \langle D_1H_t( \hat{x}_t, \hat{u}_t, p_{t+1}), \hat{x}_t - x_t \rangle\\
\null & \null & - \langle D_1H_{t+1}( \hat{x}_{t+1}, \hat{u}_{t+1}, p_{t+2}), \hat{x}_{t+1} - x_{t+1} \rangle ]
\end{array}
\]
Therefore, using (\ref{eq10.1}), we obtain, for all $T \in \N_*$, 
\[
\begin{array}{rcl}
\sum_{t=0}^T \beta^t \phi (\hat{x}_t, \hat{u}_t) -\sum_{t=0}^T  \beta^t \phi(x_t,u_t) &\geq &\langle D_1 H_0(\sigma, \hat{u}_0, p_1), \sigma - \sigma \rangle \\
\null & \null &- \langle p_{T+1}, \hat{x}_{T+1} - x_{T+1} \rangle \Longrightarrow
\end{array}
\]
\begin{equation}\label{eq10.5}
\sum_{t=0}^T \beta^t \phi (\hat{x}_t, \hat{u}_t) -\sum_{t=0}^T  \beta^t \phi(x_t,u_t) \geq - \langle p_{T+1}, \hat{x}_{T+1} - x_{T+1} \rangle.
\end{equation}
Since $\underline{p} \in {\ell}^1(\N_*, \R^{n*})$, we have $\lim_{T \rightarrow + \infty} p_{T+1} = 0$, and since $\hat{\underline{x}}, \underline{x} \in c_0(\N, \R^n)$ we have 
 $\lim_{T \rightarrow + \infty} (\hat{x}_{T+1} - x_{T+1}) = 0$ which implies $\lim_{T \rightarrow + \infty}( - \langle p_{T+1}, \hat{x}_{T+1} - x_{T+1} \rangle) = 0$, and then, from 
(\ref{eq10.5}), doing $T \rightarrow + \infty$ we obtain $J(\hat{\underline{x}}, \hat{\underline{u}}) - J (\underline{x}, \underline{u}) \geq 0$. And so we have proven that $(\hat{\underline{x}}, \hat{\underline{u}})$ is a solution of (P1).
\end{proof}
\begin{remark}\label{rem102}
The structure of the previous proof is inspired by the proof of Theorem 5.1 in \cite{BH2}. Note that our assumption (iii) permits to avoid to assume that $U$ is compact. Moreover note that we can replace the assumption (iii) by the condition: $U$ is closed.
\end{remark}
\begin{remark}\label{rem103}
Note that under our assumptions, the process $(\hat{\underline{x}}, \hat{\underline{u}})$ is also solution of the following  problem
\[
\left\{
\begin{array}{cl}
{\rm Mawimize} & \sum_{t=0}^{+ \infty} \beta^t \phi(x_t,u_t)\\
{\rm when}& \underline{x} \in {\ell}^{\infty}(\N, \R^n), \underline{u} \in {\ell}^{\infty}(\N, U)\\
{\rm and}& \forall t \in \N, x_{t+1} = f(x_t,u_t), x_0 = \sigma
\end{array}
\right.
\]
since, in the previous proof, when we obtain (\ref{eq10.5}), having $\hat{\underline{x}}$ and  $\underline{x}$ bounded is sufficient to obtain $\lim_{T \rightarrow + \infty}( - \langle p_{T+1}, \hat{x}_{T+1} - x_{T+1} \rangle) = 0$ and consequently to have the optimality of $(\hat{\underline{x}}, \hat{\underline{u}})$ for the last problem.
\end{remark}

\section{Sufficient conditions for (P)}
This section is devoted to the translation of the result of sufficient condition of optimality on (P1) into an analogous result on (P).\\
When $y_{\infty} \in \R^n$, we denotes by $c_{y_{\infty}}(\N, \R^n)$ the set of the sequences $\underline{y}$ in $\R^n$ such that $\lim_{t \rightarrow + \infty} y_t = y_{\infty}$. It is a complete affine subset of ${\ell}^{\infty}(\N, \R^n)$.
\vskip1mm
\begin{theorem}\label{th111}
Let $U$ be a nonempty convex subset of $\R^d$, $\beta \in (0,1)$, $\eta, y_{\infty} \in \R^n$, and two mappings $\psi : \R^n \times U \rightarrow \R$ and $g : \R^n \times U \rightarrow \R^n$. Let  $(\hat{\underline{y}}, \hat{\underline{u}}) \in  c_{y_{\infty}}(\N, \R^n) \times {\ell}^{\infty}(\N, U)$ and $\underline{p} \in {\ell}^1(\N_*, \R^{n*})$ which satisfy the following conditions.
\begin{itemize}
\item[(i)] For all $t \in \N$, $\hat{y}_{t+1} = g(\hat{y}_t, \hat{u}_t)$, and $\hat{y}_0 = \eta$.
\item[(ii)]  $\psi \in C^1(\R^n \times U, \R)$ and $g \in C^1(\R^n \times U, \R^n)$.
\item[(iii)] $\psi$ transforms bounded subsets of $\R^n \times U$ into bounded subsets of $\R$.
\item[(v)]  $p_t = p_{t+1} \circ D_1g(\hat{y}_t, \hat{u}_t) + \beta^t D_1 \psi (\hat{y}_t, \hat{u}_t)$  for all $t \in \N_*$.
\item[(vi)]  $\langle p_{t+1} \circ D_2 g(\hat{y}_t, \hat{u}_t) + \beta^t D_2 \psi(\hat{y}_t, \hat{u}_t) , u - \hat{u} \rangle \leq 0$ for all $u \in U$, for all $t \in \N$
\item[(vii)] The function $[(y,u) \mapsto \langle p_{t+1}, g(y, u) \rangle  + \beta^t \psi (y,u)]$ is concave on $\R^n \times U$ for all $t \in \N$.
\end{itemize}
Then $(\hat{\underline{y}}, \hat{\underline{u}})$ is a solution of (P).
\end{theorem}
\begin{proof}
Using Section 2, $\hat{x}_t = \hat{y}_t - y_{\infty}$ for all $t \in \N$, we see that $(\hat{\underline{x}}, \hat{\underline{u}}) \in c_0(\N, \R^n) \times {\ell}^{\infty}(\N, U)$ satisfies all the assumptions of Theorem \ref{th101}. And so $(\hat{\underline{x}}, \hat{\underline{u}})$ is a solution of (P1) which implies that $(\hat{\underline{y}}, \hat{\underline{u}})$ is a solution of (P).
\end{proof} 

\begin{thebibliography}{00}
%
\bibitem{ATF} V.M. Alex\'eev, V.M. Tihomirov and S.V. Fomin, {\sf Commande optimale}, French edition, MIR, Moscow, 1982.
%
\bibitem{AB} C. D. Aliprantis and K. C. Border, {\sf Infinite dimensional analysis}, Second edition, Springer-Verlag, Berlin, 1999.
%
\bibitem{BCr} J. Blot and B. Crettez, {\it On the smoothness of optimal paths}, Decision Economics Finance, {\bf 27}, 2004, 1-34.
%
\bibitem{BH1}  J. Blot and  N. Hayek, {\it Infinite horizon discrete time control problems for bounded processes}, Adv. Difference Equ. 2008; 2008:1-15. Article ID 654267.
%
\bibitem{BH2} J. Blot and  N. Hayek, {\sf Infinite-horizon optimal control in the discrete-time framework}, Springer, New York, 2014.
%
\bibitem{BHPP} J. Blot, N. Hayek, F. Pekergin and N. Pekergin, {\it Pontryagin principles for bounded discrete-time processes}, Optimization, {\bf}(3), 2015, 505-520.
%
\bibitem{Br} H. Brezis, {\sf Functional analysis, Sobolev spaces and partial differential equations}, SpringerScience+Business Media, LLC, New York, 2011.
%
\bibitem{Co} F. Colonius, {\sf Optimal periodic control}, Lect. Notes Math. $n^0 1313$, Springer-Verlag, Berlin, 1988.
%
\bibitem{IT} A.D. Ioffe and V.M. Tihomirov, {\sf Theory of extremal problems}, English edition, North-Holland Publishing Company, Amsterdam, 1979.
%
\bibitem{Ja} J. Jahn, {\sf Introduction of the theory of nonlinear optimization}, Third edition,  Springer-Verlag, Berlin, 2007.
%
\bibitem{Mi} J W. Milnor, {\sf Topology from the differentiable viewpoint}, Sixth printing, The University Press of Virginia, Charlottesville, 1969.
%
\bibitem{Sp} M. Spivak, {\it Calculus on manifolds}, W.A. Benjamin, Inc., New York, 1965.
%
\bibitem{We} J. Werner, {\sf Optimization theory and applications}, Vieweg, Braunschweg / Wiesbaden, 1984.
%
\end{thebibliography}
\end{document}